\documentclass[11pt]{article}
\usepackage[pagebackref=true,colorlinks,linkcolor=blue,citecolor=magenta]{hyperref}
\usepackage{amssymb,amsmath}
\usepackage{amsmath}
\usepackage{dsfont}
\usepackage{amsfonts}
\newtheorem{precor}{{\bf Corollary}}

\newenvironment{cor}{\begin{precor}{\hspace{-0.5
               em}{\bf.\ }}}{\end{precor}}
\newtheorem{precon}{{\bf Conjecture}}

\newtheorem{prealphcon}{{\bf Conjecture}}

\newenvironment{alphcon}{\begin{prealphcon}{\hspace{-0.5
               em}{\bf.\ }}}{\end{prealphcon}}
\newtheorem{predefin}{{\bf Definition}}

\newtheorem{preexm}{{\bf Example}}

\newtheorem{preappl}{{\bf Application}}

\newtheorem{prelem}{{\bf Lemma}}

\newenvironment{lem}{\begin{prelem}{\hspace{-0.5
               em}{\bf.\ }}}{\end{prelem}}
\newtheorem{preproof}{{\bf Proof.\ }}

\newenvironment{proof}[1]{\begin{preproof}{\rm
               #1}\hfill{$\blacksquare$}}{\end{preproof}}
\newtheorem{prethm}{{\bf Theorem}}

\newenvironment{thm}{\begin{prethm}{\hspace{-0.5
               em}{\bf.\ }}}{\end{prethm}}
\newtheorem{prealphthm}{{\bf Theorem}}

\newenvironment{alphthm}{\begin{prealphthm}{\hspace{-0.5
               em}{\bf.\ }}}{\end{prealphthm}}

\newtheorem{prealphlem}{{\bf Lemma}}

\newenvironment{alphlem}{\begin{prealphlem}{\hspace{-0.5
               em}{\bf.\ }}}{\end{prealphlem}}

\newtheorem{prepro}{{\bf Proposition}}

\newenvironment{pro}{\begin{prepro}{\hspace{-0.5
               em}{\bf.\ }}}{\end{prepro}}
\newtheorem{preprb}{{\bf Problem}}

\newtheorem{prerem}{{\bf Remark}}

\newtheorem{preapp}{{\bf Application}}

\newtheorem{prequ}{{\bf Question}}

\newtheorem{preclaim}{{\bf Claim}}

\def\conct[#1,#2]{\mbox {${#1} \leftrightarrow {#2}$}}
\def\dconct[#1,#2]{\mbox {${#1} \rightarrow {#2}$}}
\def\deg[#1,#2]{\mbox {$d_{_{#1}}(#2)$}}
\def\mindeg[#1]{\mbox {$\delta_{_{#1}}$}}
\def\maxdeg[#1]{\mbox {$\Delta_{_{#1}}$}}
\def\outdeg[#1,#2]{\mbox {$d_{_{#1}}^{^+}(#2)$}}
\def\minoutdeg[#1]{\mbox {$\delta_{_{#1}}^{^+}$}}
\def\maxoutdeg[#1]{\mbox {$\Delta_{_{#1}}^{^+}$}}
\def\indeg[#1,#2]{\mbox {$d_{_{#1}}^{^-}(#2)$}}
\def\minindeg[#1]{\mbox {$\delta_{_{#1}}^{^-}$}}
\def\maxindeg[#1]{\mbox {$\Delta_{_{#1}}^{^-}$}}

\def\dre[#1,#2,#3]{\mbox {${\cal E}^{^{#3}}(#1,#2)$}}
\def\var[#1,#2]{\mbox {${\rm Var}_{_{#1}}(#2)$}}
\def\ls[#1]{\mbox {$\xi^{^{#1}}$}}
\def\hom[#1,#2]{\mbox {${\rm Hom}({#1},{#2})$}}
\def\onvhom[#1,#2]{\mbox {${\rm Hom^{v}}(#1,#2)$}}
\def\onehom[#1,#2]{\mbox {${\rm Hom^{e}}(#1,#2)$}}
\def\core[#1]{\mbox {$#1^{^{\bullet}}$}}
\def\cay[#1,#2]{\mbox {${\rm Cay}({#1},{#2})$}}
\def\sch[#1,#2,#3]{\mbox {${\rm Sch}({#1},{#2},{#3})$}}
\def\cays[#1,#2]{\mbox {${\rm Cay_{s}}({#1},{#2})$}}
\def\dirc[#1]{\mbox {$\stackrel{\rightarrow}{C}_{_{#1}}$}}
\def\cycl[#1]{\mbox {${\bf Z}_{_{#1}}$}}

\begin{document}
\footnotetext[1]{The research of Hossein Hajiabolhassan is supported by ERC advanced grant GRACOL.}
\begin{center}
{\Large \bf On Chromatic Number of Kneser Hypergraphs}\\
\vspace{0.3 cm}
{\bf Meysam Alishahi$^\dag$ and Hossein Hajiabolhassan$^\ast$\\
{\it $^\dag$ Department of Mathematics}\\
{\it Shahrood University of Technology, Shahrood, Iran}\\
{\tt meysam\_alishahi@shahroodut.ac.ir}\\
{\it $^\ast$ Department of Applied Mathematics and Computer Science}\\
{\it Technical University of Denmark}\\
{\it DK-{\rm 2800} Lyngby, Denmark}\\
{\it $^\ast$ Department of Mathematical Sciences}\\
{\it Shahid Beheshti University, G.C.}\\
{\it P.O. Box {\rm 19839-63113}, Tehran, Iran}\\
{\tt hhaji@sbu.ac.ir}\\
}
\end{center}
\begin{abstract}
\noindent
In this paper, in view of $Z_p$-Tucker lemma, we introduce a lower
bound for chromatic number of Kneser hypergraphs
which improves Dol'nikov-K{\v{r}}{\'{\i}}{\v{z}} bound. Next,
we introduce multiple Kneser hypergraphs and
we specify the chromatic number of some multiple Kneser hypergraphs.
For a  vector of positive integers $\vec{s}=(s_1,s_2,\ldots ,s_m)$ and a
partition $\pi=(P_1,P_2,\ldots ,P_m)$ of $\{1,2,\ldots,n\}$,
the multiple Kneser hypergraph ${\rm KG}^r(\pi; \vec{s};k)$ is a
hypergraph with the vertex set
$$V=\left\{A:\ A\subseteq P_1\cup P_2\cup\cdots \cup P_m,\ |A|=k,
\forall 1\leq i\leq m;\ |A\cap P_i|\leq s_i\right\}$$
whose edge set is consist of any $r$ pairwise disjoint vertices.
We determine the chromatic number of multiple Kneser hypergraphs
provided that $r=2$ or for any $1\leq i\leq m$, we have $|P_i|\leq 2s_i$.
In particular, one can see that if $|P_1|=|P_2|=\cdots=|P_m|=t$, $m\geq k$, and
$\vec{s}=(1,1,\ldots ,1)$, then $\chi({\rm KG}^2(\pi; \vec{s};k))= t(m-k+1)$.
This gives a positive answer to a problem of Naserasr and Tardif
[The chromatic covering number of a graph, Journal of Graph Theory, 51 (3): 199--204, (2006)].

A subset $S \subseteq [n]$ is almost $s$-stable if for any two distinct elements $i,j\in S$,
we have $|i-j|\geq s$.
The almost $s$-stable Kneser hypergraph ${\rm KG}^r(n,k)_{s-stab}^{\sim}$
has all $s$-stable subsets of $[n]$ as the vertex set and every $r$-tuple of
pairwise disjoint vertices forms an edge. Meunier [The chromatic number of almost stable
Kneser hypergraphs. J. Combin. Theory Ser. A, 118(6):1820--1828, 2011] showed for any positive integer $r$, $\chi({\rm KG}^r(n,k)_{2-stab}^{\sim})=\left\lceil {n-r(k-1) \over r-1}\right\rceil$. We extend this result to a large family of Schrijver hypergraphs. Finally, we present a colorful-type result which confirms   the existence of a completely multicolored complete bipartite graph in any coloring of a graph.

\noindent {\bf Keywords:}\ { Chromatic Number, Kneser Hypergraph, $Z_p$-Tucker Lemma, Tucker-Ky Fan's Lemma.}\\
{\bf Subject classification: 05C15}
\end{abstract}
\section{Introduction}
 In this section, we setup some notations and terminologies.
 Hereafter, the symbol $[n]$ stands for the set $\{1,\ldots, n\}$.
 A {\it hypergraph} ${\cal H}$ is an ordered pair $({\cal V(H)},{\cal E(H)})$,
 where ${\cal V(H)}$ (the {\it vertex set}) is a finite set and
 ${\cal E(H)}$ (the \emph{edge set}) is a family of distinct non-empty
 subsets of ${\cal V(H)}$. Throughout this paper,
 we suppose that ${\cal V(H)}=[n]$ for some positive integer $n$.
 Assume that $N=(N_1,N_2,\ldots,N_r)$, where $N_i$'s are pairwise disjoint subsets of $[n]$.
 The \emph{induced hypergraph} ${\cal H}_{|_N}$ has $\displaystyle \cup_{i=1}^rN_i$ and
 $\{A\in {\cal E(H)}:\ \exists i;\ 1\leq i\leq r, A\subseteq N_i\}$ as
 the vertex set and the edge set, respectively.
 If every edge of hypergraph has size $r$, then it is called an {\it $r$-uniform} hypergraph.
 A {\it $t$-coloring} of a hypergraph ${\cal H}$ is a mapping
 $h:{\cal V(H)}\longrightarrow [t]=\{1,2,\ldots,t\}$ such that every edge is~not monochromatic.
 The minimum $t$ such that there exists a $t$-coloring for
hypergraph ${\cal H}$ is termed its {\it chromatic number}, and is denoted by $\chi({\cal H})$.
If ${\cal H}$ has an edge of size 1, then we define the chromatic number of ${\cal H}$ to be infinite.
For any hypergraph ${\cal H}=({\cal V(H)},{\cal E(H)})$ and positive integer $r\geq 2$,
the \emph{Kneser hypergraph}
${\rm KG}^r({\cal H})$ has ${\cal E(H)}$ as its vertex set and  the edge
set consisting of all $r$-tuples of pairwise
disjoint edges of ${\cal H}$. It is known that for any graph $G$, there exists a
hypergraph ${\cal H}$ such that $G$ is isomorphic to ${\rm KG}^2({\cal H})$.

A subset $S \subseteq [n]$ is {\it $s$-stable} (reps. {\it almost $s$-stable})
if any two distinct elements of $S$ are at least
``at distance $s$ apart'' on the $n$-cycle (reps. $n$-path),
that is, $s\leq |i-j|\leq n-s$ (reps. $|i-j|\geq s$)
for distinct $i,j\in S$. Hereafter, for a subset $A\subseteq [n]$,
the symbols ${A\choose k}$, ${A\choose k}_s$, and ${A\choose k}_s^{\sim}$ stand
for the set of all $k$-subsets of $A$, the set of all $s$-stable $k$-subsets of $A$,
and the set of all almost $s$-stable $k$-subsets of $A$, respectively. One can see that
${A\choose k}_s\subseteq {A\choose k}_s^{\sim} \subseteq {A\choose k}$.
Assume that ${\cal H}_1=(A,{A\choose k})$, ${\cal H}_2=(A,{A\choose k}_s)$,
and ${\cal H}_3=(A,{A\choose k}_s^{\sim})$. Hereafter, for any positive integer $r\geq 2$,  the hypergraphs
${\rm KG}^r({\cal H}_1)$, ${\rm KG}^r({\cal H}_2)$, and  ${\rm KG}^r({\cal H}_3)$
are denoted by ${\rm KG}^r(A,k)$,
${\rm KG}^r(A,k)_{s-stab}$, and ${\rm KG}^r(A,k)_{s-stab}^{\sim}$, respectively. Also,
${\rm KG}^r([n],k)$, ${\rm KG}^r([n],k)_{s-stab}$, and ${\rm KG}^r([n],k)_{s-stab}^{\sim}$,
are denoted by ${\rm KG}^r(n,k)$,
${\rm KG}^r(n,k)_{s-stab}$, and ${\rm KG}^r(n,k)_{s-stab}^{\sim}$, respectively.
Moreover, hypergraphs ${\rm KG}^r(n,k)$,
${\rm KG}^r(n,k)_{s-stab}$, and ${\rm KG}^r(n,k)_{s-stab}^{\sim}$ are termed the {\it Kneser
hypergraph}, the {\it $s$-stable Kneser hypergraph},
and the {\it almost $s$-stable Kneser hypergraph}, respectively.
In 1955, Kneser \cite{MR0068536} conjectured $\chi({\rm KG}^2(n,k))= n-2k+2$.
Later, Lov\'{a}sz \cite{MR514625}, in  his fascinating paper, confirmed the
conjecture using algebraic topology. Next, Erd{\H o}s \cite{MR0465878} presented an upper bound for the
chromatic number of Kneser hypergraphs and conjectured the equality. In~\cite{MR857448},
this conjecture has been confirmed and it was shown
$\chi({\rm KG}^r(n,k))=\left\lceil {n-r(k-1) \over r-1}\right\rceil$.
In view of result of Erd{\H o}s \cite{MR0465878},
one can conclude that for any positive integer $r\geq 2$, we have
$\chi({\rm KG}^r(n,k)_{s-stab})\leq \chi({\rm KG}^r(n,k)_{s-stab}^{\sim})\leq
\left\lceil {n-r(k-1) \over r-1}\right\rceil$. Also, Meunier \cite{MR2793613}
has shown that when $s\geq r$, then $\chi({\rm KG}^r(n,k)_{s-stab})\leq \left\lceil {n-s(k-1)
\over r-1}\right\rceil$. Finding a lower bound for chromatic number of
hypergraphs has been studied in the literature, see
\cite{MR953021,MR1081939,MR1665335,MR2279672,MR2351519,MR1893009,MR2208423}.
As an interesting result, we have Dol'nikov-K{\v{r}}{\'{\i}}{\v{z}} bound. For a hypergraph ${\cal H}$,
the {\it $r$-colorability defect} of ${\cal H}$, say $cd_r({\cal H})$,
is the minimum number of vertices which should be excluded such that
the induced hypergraph on the remaining vertices has chromatic number at most $r$.

\begin{alphthm}{\rm (}Dol'nikov for $r=2$, K{\v{r}}{\'{\i}}{\v{z}} {\rm \cite{MR953021,MR1081939}}{\rm )}\label{DOLN}
For any hypergraph ${\cal H}$ and positive integer $r\geq 2$,
we have $$\chi({\rm KG}^r({\cal H}))\geq {cd_r({\cal H})\over r-1}.$$
\end{alphthm}

Alon et al. \cite{MR2448565}  constructed ideals in $\Bbb N$ which are not non-atomic but
they have the Nikod\'ym property by using stable Kneser hypergraphs.
In this regard, they studied the chromatic number of
$r$-stable Kneser hypergraph ${\rm KG}^r(n,k)_{r-stab}$ and presented the following conjecture.

\begin{alphcon}{\rm \cite{MR2448565}}\label{ASTABLE}
Let $k, r$ and $n$ be positive integers where $n\geq rk$ and $r\geq 2$. We have
$$\chi({\rm KG}^r(n,k)_{r-stab})= \left\lceil {n-r(k-1) \over r-1}\right\rceil.$$
\end{alphcon}

In \cite{MR2448565}, it has been shown that the aforementioned conjecture holds when $r$ is a power of $2$.
As an approach to Conjecture \ref{ASTABLE}, Meunier~\cite{MR2793613}
showed that $\chi({\rm KG}^r(n,k)_{2-stab}^{\sim})=\left\lceil {n-r(k-1) \over r-1}\right\rceil$ and he strengthened the above conjecture as follows.

\begin{alphcon}{\rm \cite{MR2793613}}
Let $k, r, s$ and $n$ be positive integers where $n\geq sk$ and $s\geq r\geq 2$. We have
$$\chi({\rm KG}^r(n,k)_{s-stab})= \left\lceil {n-s(k-1) \over r-1}\right\rceil.$$
\end{alphcon}

In \cite{jonsson}, it has been shown that if $r=2$ and $n$ is sufficiently large, then for $s\geq 4$ the aforementioned conjecture holds.

This paper is organized as follows. In Section 1, we set up notations and terminologies.
Section 2 will be concerned with Tucker's lemma and its generalizations.
In Section $3$, we introduce the $i^{\rm th}$ alternation number of hypergraphs, and
in view of $Z_p$-Tucker lemma, we present a lower bound for chromatic number of hypergraphs.
This lower bound improves the well-know Dol'nikov-K{\v{r}}{\'{\i}}{\v{z}}
lower bound (Theoreme~\ref{DOLN}). In fact, the alternation number of a
hypergraph can be considered as a generalization of colorability defect of hypergraphs.
Section 4 is devoted to multiple Kneser hypergraphs.
In this section, we determine the chromatic number of some
multiple Kneser hypergraphs. In particular, we present a generalization of a result of
Alon et al. \cite{MR857448} about chromatic number of Kneser hypergraphs.
In Section 5, we extend Meunier's result and it is shown that for any positive integers $k, n$ and $r$, if $r$ is an even integer
or $n\stackrel{r-1}{\not \equiv} k $,
then for \emph{Schrijver hypergraph} ${\rm KG}^r(n,k)_{2-stab}$,
we have $\chi({\rm KG}^r(n,k)_{2-stab})= \left\lceil {n-r(k-1) \over r-1}\right\rceil$.
In the last Section, in view of Tucker-Ky Fan's Lemma,
we prove a colorful-type result which confirms the existence of a completely
multicolored complete bipartite graph in any coloring of a graph in terms of its alternation number.
\section{Tucker's Lemma and Its Generalizations}
In this section, we present Tucker's lemma and some of its generalizations.
In fact, Tucker's lemma is a combinatorial version of the
Borsuk-Ulam theorem with several interesting applications.
For more about Bursuk-Ulam's theorem and Tucker's lemma, see \cite{MR1988723}.

Throughout this paper, for any positive integer $r$, let ${\cal Z}_r=\{\omega_1,\omega_2,\ldots,\omega_r\}$
be a set of size $r$ where $0\not \in {\cal Z}_r$. Moreover, when $p$ is a prime integer,
we assume that ${\cal Z}_p$ is a \emph{cyclic group} of order $p$ and
\emph{generator} $\omega$, i.e., ${\cal Z}_p=Z_p=\{\omega,\omega^2,\ldots,\omega^p\}$.  In particular, when
$r=2$, we set ${\cal Z}_2=Z_2=\{\omega,\omega^2\}=\{-1,+1\}$.
Consider a \emph{ground set} $S$ such that $0\not \in S$ and $|S|\geq 2$.
Assume that  $X=(x_1,x_2,\ldots,x_n)$ is a sequence of $S\cup\{0\}$.
The subsequence $x_{j_1}, x_{j_2},\ldots,x_{j_m}$ (${j_1}<{j_2}<\cdots<{j_m}$)
is said to be an {\it alternating sequence} if any two consecutive terms in this subsequence are different.
We denote by $alt(X)$ the size of a longest alternating subsequence of non-zero terms in $X$.
For instance, if $S={\cal Z}_4$ and
$X=(\omega_4,0,\omega_2, \omega_1,0,\omega_1,\omega_3,\omega_1)$, then $alt(X)=5$.

One can consider $({\cal Z}_r\cup \{0\})^n$ as the set of all
{\it singed subsets} of $[n]$, that is, the
family of all $(X^1,X^2,\ldots,X^r)$ of disjoint subsets of $[n]$. Precisely,
for $X=(x_1,x_2,\ldots,x_n)\in({\cal Z}_r\cup \{0\})^n\setminus \{(0,0,\ldots,0)\}$ and $j\in[r]$, we define
$X^j=\{i\in [n]:\ x_i=\omega_j\}$. Throughout of this paper, for any $X\in ({\cal Z}_r\cup \{0\})^n$,
we use these representations interchangeably, i.e., $X=(x_1,x_2,\ldots,x_n)$ or $X=(X^1,X^2,\ldots,X^r)$.
Assume $X,Y\in ({\cal Z}_r\cup\{0\})^n$. By $X\preceq Y$,
we mean $X^i\subseteq Y^i$ for each $i\in[r]$.
Note that if $X\preceq Y$, then any alternating subsequence of $X$
is also an alternating subsequence of $Y$ and therefore $alt(X)\leq alt(Y)$.
Also, note that if the first non-zero term in $X$ is $\omega_j$,
then any alternating subsequence of $X$ of maximum length begins with $\omega_j$ and also,
we can conclude that $X^j$  contains the smallest integer.
For a permutation $\sigma$ of $[n]$, by $alt_\sigma(X)$,
we denote the length of a longest alternating subsequence of non-zero signs in
$(x_{\sigma(1)},x_{\sigma(2)},\ldots,x_{\sigma(n)})$, i.e., $alt((x_{\sigma(1)},x_{\sigma(2)},\ldots,x_{\sigma(n)}))$. In this terminology,
$alt(X)$ is the same as $alt_I(X)$, where $I$ is the identity permutation.
Now, we are in a position to introduce Tucker's lemma.

\begin{alphlem}
{\rm(}Tucker's lemma {\rm \cite{MR0020254} )} Suppose that $n$ is a positive integer and
$\lambda:\{-1,0,+1\}^n\setminus \{(0,\ldots,0)\} \longrightarrow \{\pm 1, \pm 2,\ldots ,\pm (n-1)\}$.
Also, assume that
for any signed set $X\in \{-1,0,+1\}^n\setminus \{(0,\ldots,0)\}$, we have $\lambda(-X)=-\lambda(X)$.
Then, there exist two signed sets
$X$ and $Y$ such that $X\preceq Y$ and also $\lambda(X) =-\lambda(Y)$.
\end{alphlem}

There exist several interesting applications of Tucker's lemma in combinatorics.
Among them, one can consider a combinatorial proof for Lov{\'a}sz-Kneser's
Theorem by Matousek \cite{MR2057690}. Also, there are various generalizations of
Tucker's lemma. Next lemma  is a combinatorial variant of $Z_p$-Tucker Lemma proved and modified
in \cite{MR1893009} and \cite{MR2793613}, respectively.
\begin{alphlem}{\rm (}$Z_p$-Tucker Lemma{\rm)}
Suppose that $n, m, p$ and $\alpha$ are non-negative integers, where $m,n\geq 1$, $m\geq \alpha \geq 0$, and $p$
is a prime number. Also, let
$$
\begin{array}{rl}
  \lambda:\ (Z_p\cup\{0\})^n\setminus\{(0,0,\ldots,0)\} & \longrightarrow Z_p\times[m] \\
  X & \longmapsto(\lambda_1(X),\lambda_2(X))
\end{array}
$$
be a map satisfying the following properties:
\begin{itemize}
\item $\lambda$ is a $Z_p$-equivariant map, that is, for each $\omega^j\in Z_p$,
we have $\lambda(\omega^jX)=(\omega^j\lambda_1(X),\lambda_2(X))$;
\item for all $X_1\preceq X_2\in (Z_p\cup\{0\})^n\setminus\{(0,0,\ldots,0)\}$,
      if $\lambda_2(X_1)=\lambda_2(X_2)\leq\alpha$, then $\lambda_1(X_1)=\lambda_1(X_2)$;
\item for all $X_1\preceq X_2\preceq\cdots \preceq X_p\in (Z_p\cup\{0\})^n\setminus\{(0,0,\ldots,0)\}$,
      if $\lambda_2(X_1)=\lambda_2(X_2)=\cdots=\lambda_2(X_p)\geq\alpha+1$,
      then the $\lambda_1(X_i)$'s are not pairwise distinct for $i=1,2,\ldots,p$.
\end{itemize}
Then $\alpha+(m-\alpha)(p-1) \geq n$.
\end{alphlem}

Another interesting generalization of the Borsuk-Ulam theorem is
Ky~Fan's lemma \cite{MR0051506}. This lemma has been used in some papers to
study some coloring properties of graphs, see \cite{MR2763055,MR2837625} .

\begin{alphlem}{\rm (}Tucker-Ky Fan's lemma {\rm \cite{MR0051506})}
Assume that $m$ and $n$ are positive integers and
$\lambda:\{-1,0,+1\}^n\setminus \{(0,\ldots,0)\} \longrightarrow \{\pm 1, \pm 2,\ldots ,\pm m\}$
satisfying the following properties:
\begin{enumerate}
\item for any $X\in \{-1,0,+1\}^n\setminus \{(0,\ldots,0)\}$, we have $\lambda(-X)=-\lambda(X)$ {\rm (}a
      $Z_2$-equivariant map{\rm )}
\item there are no any two signed sets
$X$ and $Y$ such that $X\preceq Y$ and $\lambda(X) =-\lambda(Y)$.
\end{enumerate}
Then there are $n$ signed sets
$X_1\preceq X_2\preceq \cdots \preceq X_n$ such that $\{\lambda(X_1),\ldots,\lambda(X_n)\}=\{+a_1,-a_2\ldots ,
(- 1)^{n-1}a_n\}$ where  $1\leq a_1 < \cdots < a_n \leq m$. In
particular $m \geq n$.
\end{alphlem}
\section{An Improvement of Dol'nikov-K{\v{r}}{\'{\i}}{\v{z}} Theorem}
For a hypergraph ${\cal F}\subseteq 2^{[n]}$, a permutation $\sigma$ of $[n]$
(a \emph{linear ordering} of $[n]$),
and positive integers $i$ and $r\geq 2$,
set $alt_{r,\sigma}({\cal F}, i)$ to be the largest integer $k$ such that there
exists an $X\in({\cal Z}_r\cup\{0\})^n\setminus \{(0,0,\ldots,0)\}$
with $alt_\sigma(X)=k$  and that the chromatic number of
hypergraph ${\rm KG}^r({\cal F}_{|X})$ is at most $i-1$.
Indeed, $alt_{r,\sigma}({\cal F},1)$ is the largest integer $k$ such that there
exists an $X=(X^1,X^2,\ldots,X^r)\in({\cal Z}_r\cup\{0\})^n\setminus \{(0,0,\ldots,0)\}$
with $alt_\sigma(X)=k$ and none of the $X^j$'s contain any member of ${\cal F}$. For instance, one can see that
$alt_{r,I}({[n]\choose k}_{2})=r(k-1)+1$. Also, hereafter, $alt_{r,\sigma}({\cal F}, 1)$ is denoted by
$alt_{r,\sigma}({\cal F})$.
Now, set $alt_r({\cal F},i)=\min\{alt_{r,\sigma}({\cal F},i);\ \sigma\in S_n\}$.
Also, $alt_r({\cal F},i)$ is
termed the $i^{th}$ {\it alternation number} of ${\cal F}$ (with respect to $r$) and
the {\it first alternation number} of ${\cal F}$ is denoted by $alt_r({\cal F})$.
In this terminology, one can see that if $i>\chi({\rm KG}^r({\cal F}))$, then
$alt_r({\cal F},i)=n$.

For a hypergraph ${\cal F}\subseteq 2^{[n]}$,  define $M_r({\cal F})$ to be the maximum size of a set
$T\subseteq [n]$ such that
the induced hypergraph on ${\cal F}_{|_T}$, is an $r$-colorable hypergraph. One can see that
$cd_r({\cal F})=n-M_r({\cal F})$. In view of Theorem \ref{DOLN},
we know $\chi({\rm KG}^r({\cal F}))\geq {cd_r({\cal F})\over r-1}={n-M_{r}({\cal F})\over r-1}$.
In this section, we show that
$\chi({\rm KG}^r({\cal F}))\geq \frac{n-alt_r({\cal F})}{r-1}$.
In fact, Theorem~\ref{DOLN} is an immediate consequence of this result.
To see this, one can check that for any permutation $\sigma$ of $[n]$, we have
$M_r({\cal F}) \geq alt_{r,\sigma}({\cal F})$. According to the definition
of $alt_{r,\sigma}({\cal F})$, there is an
$X\in({\cal Z}_r\cup\{0\})^n\setminus\{(0,0,\ldots,0)\}$, such that none of
$X^k$'s ($1\leq k\leq n$) contain any member of ${\cal F}$ and that $alt_\sigma(X)=alt_{r,\sigma}({\cal F})$.
Set $T=\displaystyle\cup_{a=1}^{r} X^a$. Obviously, $(X^1,\ldots,X^r)$ is a proper
$r$-coloring of ${\cal F}_{|_T}$.
This implies that $|X|\leq M_r({\cal F})$. On the other hand,
$alt_\sigma(X)\leq |X|$ and therefore
$$\chi({\rm KG}^r({\cal F}))\geq {n-alt_{r,\sigma}({\cal F})\over r-1}
\geq {n-|X|\over r-1}\geq {n-M_r({\cal F})\over r-1}={cd_r({\cal F})\over r-1}.$$

Assume that $h$ is a proper coloring of $H={\rm KG}^r({\cal F})$ with colors $\{1,2,\ldots,t\}$
where $\varnothing\not={\cal F}\subseteq 2^{[n]}$.
For any subset $B\subseteq [n]$, we define $\bar{h}(B)=\max\{h(S):\ S\subseteq B,\ S\in {\cal F}\}$, if there is no $S\subseteq B$ where $S\in {\cal F}$,
then set $\bar{h}(B)=0$. For any $X=(X^1,X^2,\ldots,X^r)\in ({\cal Z}_r\cup\{0\})^n\setminus\{(0,0,\ldots,0)\}$, define
$\bar{h}(X)=\max\{\bar{h}(X^1),\bar{h}(X^2),\ldots,\bar{h}(X^r)\}$, i.e.,
$$\bar{h}(X)=\max\{h(A):\ A\in {\cal F}\ \&\ \exists j\in[r]\ s.t.\ A\subseteq X^j\}.$$

Now, we are ready to improve Theorem~\ref{DOLN}.
\begin{lem}\label{gendol}
Assume that ${\cal F}\subseteq 2^{[n]}$ is a hypergraph, and $p$ is a prime number.
For any positive integer $i$ where $i\leq\chi({\rm KG}^p({\cal F}))+1$, we have
$$\chi({\rm KG}^p({\cal F}))\geq \frac{n-alt_p({\cal F},i)}{p-1}+i-1.$$
\end{lem}
\begin{proof}{
Consider an arbitrary total ordering $\leqslant$ on $2^{[n]}$.
To prove the assertion, it is enough to show that for any $\sigma\in S_n$, we have
$\chi({\rm KG}^p({\cal F}))\geq \frac{n-alt_{p,\sigma}({\cal F},i)}{p-1}+i-1$.
Without loss of generality, we can suppose $\sigma=I$.
Let  ${\rm KG}^p({\cal F})$ be properly colored
with $C$ colors $\{1,2,\ldots,C\}$.
For any  $F \in {\cal F}$, we denote its color by $h(F)$.
Set $\alpha=alt_{p,I}({\cal F},i)$ and $m=alt_{p,I}({\cal F},i)+C-i+1$.

Now, define a map $\lambda:\ (Z_p\cup\{0\})^n\setminus\{(0,0,\ldots,0)\}  \longrightarrow Z_p\times[m]$ as follows
\begin{itemize}
\item If $X\in ({\cal Z}_r\cup\{0\})^n\setminus\{(0,0,\ldots,0)\}$ and
      $alt_I(X)\leq alt_{p,I}({\cal F},i)$, set $\lambda(X)=(\omega^j,alt_I(X))$,
      where $j$ is the index of set $X^j$ containing the smallest integer ($\omega^j$ is
      then the first non-zero term in $X=(x_1, . . . , x_n)$).

\item If $X\in ({\cal Z}_r\cup\{0\})^n\setminus\{(0,0,\ldots,0)\}$ and
      $alt_I(X)\geq alt_{p,I}({\cal F},i)+1$, in view of definition of $alt_{p,I}({\cal F},i)$,
      the chromatic number of ${\rm KG}^p({\cal F}_{|_{X}})$ is at least $i$.
      Set $\lambda(X)=(\omega^j,\bar{h}(X)-i+1+\alpha)$, where $j$ is a positive integer such that
      there is an $A\in{\cal F}$ where $A\subseteq X^j$,
      $h(A)=\bar{h}(X)$, and $A$ is the biggest such a subset respect to
      the total ordering $\leqslant$ (note that $\bar{h}(X)\geq i$).
\end{itemize}
One can check that $\lambda$ is a $Z_p$-equivariant map from $(Z_p\cup\{0\})^n\setminus\{(0,0,\ldots,0)\}$
to $Z_p\times [m]$.

Let $X_1\preceq X_2\in (Z_p\cup\{0\})^n\setminus\{(0,0,\ldots,0)\}$.
If $\lambda_2(X_1)=\lambda_2(X_2)\leq \alpha$, then the size of longest alternating subsequences of non-zero
terms of $X_1$ and $X_2$ are the same. Therefore, the first non-zero terms of $X_1$
and $X_2$ are equal; and equivalently, $\lambda_1(X_1)=\lambda_1(X_2)$.

Assume that $X_1\preceq X_2\preceq\cdots\preceq X_p\in (Z_p\cup\{0\})^n\setminus\{(0,0,\ldots,0)\}$
such that $\lambda_2(X_1)=\lambda_2(X_2)=\cdots=\lambda_2(X_p)\geq\alpha+1$.
According to the definition of $\lambda$,
for each $1\leq a\leq p$, there are $F_a\in{\cal F}$ and $j_a\in [p]$ such that
$F_a\subseteq X^{j_a}_a$ and $\lambda_2(X_a)=h(F_a)+i-1+\alpha$.
This implies that $h(F_1)=h(F_2)=\cdots=h(F_p)$.
If $|\{j_1,j_2,\ldots,j_p\}|=p$,
then $\{F_1,F_2,\ldots,F_p\}$ is an edge in ${\rm KG}^p({\cal F})$.
But, this is a contradiction because $h$ is a proper coloring and $h(F_1)=h(F_2)=\cdots=h(F_p)$.

Now, we can apply the $Z_p$-Tucker Lemma and conclude that
$n\leq alt_{p,I}({\cal F},i)+(C-i+1)(p-1)$ and so
$C\geq {n-alt_{p,I}({\cal F},i)\over p-1}+i-1$.
}\end{proof}
\begin{lem}\label{altrs}
Suppose that $r,s$ and $n$ are positive integers,
$X=(X^1,X^2,\ldots,X^r)\in({\cal Z}_r\cup\{0\})^n\setminus\{(0,0,\ldots,0)\}$, and $\sigma$ is a permutation of $[n]$.
Also, assume that for each $j\in[r]$,
$Y^{1j},Y^{2j},\ldots,Y^{sj}$ are disjoint subsets of $X^j$. If we set
$$Y_j=(Y^{1j},Y^{2j},\ldots,Y^{sj})\in({\cal Z}_s\cup\{0\})^{n}\setminus\{(0,0,\ldots,0)\}$$
and
$$Z=\left(Y^{11},\ldots,Y^{s1}, \ldots,Y^{1r},\ldots,Y^{sr}\right)
\in\left({\cal Z}_{rs}\cup \{0\}\right)^n\setminus\{(0,0,\ldots,0)\},$$
then $alt_\sigma(Z)\geq\displaystyle \sum_{i=1}^r alt_{\sigma|_{X^i}}(Y_i)$.
\end{lem}
\begin{proof}{
Without loss of generality, we can suppose $\sigma=I$.
Also, let $alt_I(X)=t$. If
$x_{a_1},x_{a_2},\cdots,x_{a_t}$ form an alternating subsequence of
$X$ ($1\leq a_1<a_2<\cdots<a_t\leq n$), then the set $\{a_1, a_2, \ldots, a_t\}$ is
called the \emph{index set} of this alternating subsequence. Choose an alternating subsequence
$x_{a_1},x_{a_2},\cdots,x_{a_t}$ of $X$ such that $a_1$ is the smallest integer
in $T=\displaystyle\cup_{j=1}^r X^j$ and that for each $i\in[t]$, there is a $j_i\in [r]$
where $[a_i,a_{i+1})\cap T\subseteq X^{j_i}$.
For each $j\in[r]$, assume that $P_j$ is a longest alternating subsequence of $Y_j$.
Now, we present an alternating subsequence $P$ of $Z$.
Construct $P$ such that for each $i\in[t]$, $P$ and $P_{j_i}$
have the same index set in $[a_i,a_{i+1})$.
It is straightforward to check that $P$ is an alternating subsequence of $Z$ and also,
$|P|=\sum_{i=1}^{r}alt_{I|_{X^i}}(Y_i)$.
}\end{proof}
Here, we extend Lemma~\ref{gendol} to any $r$ uniform hypergraph  ($r$ is not necessarily prime) for the first alternation number.
\begin{lem}\label{dolrs}
Let $r$ and $s$ be positive integers where $r,s\geq 2$. Also, assume that for any hypergraph ${\cal H}\subseteq 2^{[n]}$,
$\chi({\rm KG}^r({\cal H}))\geq {n-alt_{r}({\cal H})\over r-1}$
and $\chi({\rm KG}^s({\cal H}))\geq {n-alt_{s}({\cal H})\over s-1}$.
For any hypergraph ${\cal F}\subseteq 2^{[n]}$, we have $\chi({\rm KG}^{rs}({\cal F}))\geq {n-alt_{rs}({\cal F})\over rs-1}$.
\end{lem}
\begin{proof}{
It is enough to show that for any $\sigma\in S_n$,
$\chi({\rm KG}^{rs}({\cal F}))\geq {n-alt_{{rs},\sigma}({\cal F})\over rs-1}$.
Without loss of generality, we can suppose $\sigma=I$.
Let $K=\chi({\rm KG}^{rs}({\cal F}))$.
 On the contrary, suppose
\begin{equation}\label{Ialtrs}
{n-alt_{{rs},I}({\cal F})}>(rs-1)K.
\end{equation}
Define the hypergraph ${\cal T }\subseteq 2^{[n]}$ as follows
$${\cal T }=\left\{N\subseteq[n]:\ |N|-alt_{{s},{I_{|_N}}}({\cal F}_{|_N})>(s-1)K\right\}.$$

Now, according to the assumption of theorem and the definition of ${\cal T}$,
for each $N\in{\cal T}$, we have
$$(s-1)\chi({\rm KG}^s({\cal F}_{|_N}))\geq |N|-alt_{s,{I_{|_N}}}({\cal F}_{|_N})>(s-1)K.$$
Consequently,
\begin{equation}\label{IIaltrs}
\chi({\rm KG}^s({\cal F}_{|_N}))>K
\end{equation}
\noindent{\bf Claim}: $n-alt_{{r},I}({\cal T})> (r-1)K$.\\
Suppose, contrary to our claim, that $n-alt_{{r},I}({\cal T})\leq(r-1)K$ and so $alt_{{r},I}({\cal T})\geq n- (r-1)K$.
By definition of $alt_{{r},I}({\cal T})$, there is an
$X=(X^1,X^2,\ldots,X^r)\in({\cal Z}_r\cup\{0\})^n\setminus\{(0,0,\ldots,0)\}$
such that $alt_I(X)\geq n-(r-1)K$ and none of $X^j$'s contain any member of ${\cal T}$.
In particular, none of them lie in ${\cal T}$.
Therefore, by the definition of ${\cal T}$, we have
$$|X^i|-alt_{s,{I_{|_{X^i}}}}({\cal F}_{|_{X^i}})\leq (s-1)K.$$
It means $|X^i|-(s-1)K\leq alt_{s,{I_{|_{X^i}}}}({\cal F}_{|_{X^i}}).$
Therefore, for each $j\in[r]$, there are $s$ disjoint sets $Y^{j1},\ldots,Y^{js}\subseteq X^j$,
such that $alt_{{I_{|_{X^j}}}}(Y^{j1},\ldots,Y^{js})\geq |X^j|-(s-1)K$
and none of them contain any member of ${\cal F}_{|_{X^j}}$. In particular,
none of them contain any member of ${\cal F}$.
Set $$Z=\left(Y^{11},\ldots,Y^{1s},\ldots,Y^{r1},\ldots,Y^{rs}\right)
\in\left({\cal Z}_{rs}\cup \{0\}\right)^n\setminus\{(0,0,\ldots,0)\}.$$
By Lemma \ref{altrs},
$$\displaystyle alt_{I}(Z)\geq \sum_{j=1}^r alt_{I_{|_{X^j}}}(Y^{j1},\ldots,Y^{js})\geq
\sum_{j=1}^r \left(|X^j|-(s-1)K\right)\geq \left(\sum_1^r |X^j|\right)-r(s-1)K.$$
Note that
$\displaystyle\sum_{j=1}^r |X^j|\geq alt_{I}(X)$ and thus,
$$alt_{I}(Z)\geq alt_{I}(X)-r(s-1)K\geq n-(r-1)K-r(s-1)K=n-(sr-1)K.$$
Since $Z$ does not contain any member of ${\cal F}$, we get
$alt_{rs,I}({\cal F})\geq n-(sr-1)K$ which contradicts inequality~(\ref{Ialtrs}).
So we have proved the Claim.

By the assumption of theorem and the claim, we have
$$(r-1)\chi({\rm KG}^r({\cal T}))\geq n-alt_{r,I}({\cal T})>(r-1)K$$
and so,
\begin{equation}\label{IIIaltrs}
\chi({\rm KG}^r({\cal T}))>K.
\end{equation}

\noindent Now, consider a proper coloring $h:{\cal F}\longrightarrow [K]$ of ${\rm KG}^{rs}({\cal F})$.
By inequality~(\ref{IIaltrs}), in every $N\in{\cal T}$, there exists a color $i\in[K]$ which has been assigned to $s$ disjoint members of ${\cal F}_{|_N}$.
Now, define $h':{\cal T}\longrightarrow [K]$ such that $h'(N)$ is the maximum color which $h$ assigns
to $s$ disjoint sets in ${\cal F}_{|_N}$.
Now, according to inequality (\ref{IIIaltrs}),
there are $r$ sets $N_1,\ldots,N_r\in{\cal T}$ which are disjoint and
from $h'$ receive the same color $i_0=h'(N_j)$.
Thus we have $rs$ sets $F_{jk}\in{\cal F}$ such that $F_{jk}\subseteq N_j$
and also, they are disjoint and $h$ assigns them the same color which is a contradiction.
}\end{proof}
Next theorem is an immediate consequence of Lemma~\ref{gendol} (in the case $i=1$) and
Lemma~\ref{dolrs}.
\begin{thm}\label{dolthm}
For any hypergraph ${\cal F}\subseteq 2^{[n]}$ and positive integer $r\geq 2$, we have
$$\chi({\rm KG}^r({\cal F}))\geq \frac{n-alt_r({\cal F})}{r-1}.$$
\end{thm}
Theorem \ref{dolthm} in general is better than
Dol'nikov-K{\v{r}}{\'{\i}}{\v{z}} lower bound.
Ziegler in \cite{MR1893009,MR2208423} showed that
$cd_r({[n]\choose k}_{t})=\max\{n-tr(k-1),0\}$.
Therefore, Dol'nikov-K{\v{r}}{\'{\i}}{\v{z}} Theorem implies that
$\chi({\rm KG}^r{(n, k)}_{2-stab})\geq {\max\{n-2r(k-1),0\}\over r-1}$.
Although, one can easily see that
$alt_{r,I}({[n]\choose k}_{2})=r(k-1)+1$ and thus
by Theorem \ref{dolthm},
$$\chi({\rm KG}^r{(n, k)}_{2-stab})\geq{n-r(k-1)-1\over r-1}> {\max\{n-2r(k-1),0\}\over r-1}.$$

It is easy to see that $alt_{2,I}({[n]\choose k}_{2},2)=2k-1$.
Therefore, in view of Lemma \ref{gendol} for $i=2$, we have the next corollary.
\begin{cor}{\rm \cite{MR512648}}
If $n$ and $k$ are positive integers where $n\geq 2k$, then we have $\chi({\rm KG}^2(n,k)_{2-stab})=n-2k+2$.
\end{cor}
\section{Multiple Kneser Graphs}
Throughout this section, we assume that $k, r, n$ and $m$ are
positive integers where $r\geq 2$ and $k\geq 1$.
Furthermore, suppose that $\pi=(P_1,P_2,...,P_m)$ is a
\emph{partition} of $[n]$ and  $\vec{s}=(s_1,s_2,...,s_m)$ is a positive integer vector
where $k\leq \displaystyle\sum_{i=1}^ms_i$ and for any $1\leq i\leq m$,
we have $s_i\leq |P_i|$. The \emph{multiple Kneser hypergraph} ${\rm KG}^r(\pi; \vec{s};k)$
is a hypergraph with the vertex set
$$V=\left\{A:\ A\subseteq P_1\cup P_2\cup\cdots \cup P_m,\ |A|=k,
\forall 1\leq i\leq m;\ |A\cap P_i|\leq s_i\right\},$$
where $\{A_1,\ldots,A_r\}$ is an edge if $A_1,A_2,\ldots,A_r$ are pairwise disjoint.
In the sequel, we determine the chromatic number of multiple Kneser hypergraphs
provided that $r=2$ or for any $1\leq i\leq m$, we have $|P_i|\leq 2s_i$.
In this regard, we define the function
$f_{r,{\pi}}$ as follows
$$f_{r,{\pi}}(P_i)= \left \{
\begin{array}{ll}
rs_i & {\rm if}\ |P_i| \geq rs_i  \\
 &\\
|P_i| & {\rm otherwise}.
\end{array}\right. $$

Also, set
$$M_{r,\pi}=\max\left\{rk-1+\sum_{j=1}^t (|P_{i_j}|-f_{r,{\pi}}(P_{i_j}))\
:\ t\in \mathbb{N}\quad \& \quad \sum_{j=1}^t f_{r,{\pi}}(P_{i_j})\leq rk-1\right\}.$$

In the next theorem, we give an upper bound for the chromatic number of multiple Kneser
hypergraph ${\rm KG}^r(\pi; \vec{s};k)$ in terms of $n$ and $M_{r,\pi}$.
\begin{lem}\label{upper}
Let $k, r, n,$ and $m$ be positive integers where $r\geq 2$ and $k\geq 1$. Also,
assume that $\pi=(P_1,P_2,...,P_m)$ is a
partition of $[n]$ and $\vec{s}=(s_1,s_2,...,s_m)$ is a positive integer vector where $k\leq \displaystyle\sum_{i=1}^ms_i$.
We have $$\chi({\rm KG}^r(\pi; \vec{s};k))\leq \max\{1,\left\lceil\frac{n-M_{r,\pi}}{r-1}+1\right\rceil \}.$$
\end{lem}
\begin{proof}{
Without loss of generality, we can
suppose that $t$ is the greatest positive
integer such that the value of $M_{r,\pi}$ is attained and moreover we
suppose that $M_{r,\pi}$ is obtained by $P_m,P_{m-1},\cdots,P_{m-t+1}$, i.e.,
$M_{r,\pi}=rk-1+\sum_{j=m-t+1}^m (|P_{j}|-f_{r,{\pi}}(P_{j}))$ and
$\sum_{j=m-t+1}^m f_{r,{\pi}}(P_{j})\leq rk-1$. In view of definition of
$M_{r,\pi}$, for any $1\leq i\leq m-t$,
one can see that $rk-\displaystyle\sum_{j=m-t+1}^mf_{r,{\pi}}(P_{j})\leq f_{r,{\pi}}(P_i)$.
If $m-t=0$, then the chromatic number of multiple Kneser
hypergraph is equal to one and there is noting to prove. Hence, suppose $m-t\geq 1$.
Consider $L$ to be a subset of $P_{m-t}$ of size $rk-1-{\displaystyle \sum_{j=m-t+1}^mf_{r,{\pi}}(P_{j})}$.
Set $T=\displaystyle L\cup\left(\bigcup_{j=m-t+1}^m P_j\right)$.
Note that the size of $C=(\displaystyle\bigcup_{j=1}^{m-t}P_j)\setminus L$ is $n-M_{r,\pi}$.
For convenience, we assume $C=\{1,2,\ldots,n-M_{r,\pi}\}$.
Now, we present a proper coloring for ${\rm KG}^r(\pi; \vec{s};k)$
using $\left\lceil\frac{n-M_{r,\pi}}{r-1}\right\rceil+1$ colors.

We show that all the vertices of ${\rm KG}^r(\pi; \vec{s};k)$ which are subsets of $T$
form an independent set.
To see this, suppose therefore (reductio ad absurdum) that this is not the case and
assume that $A_1,A_2,\ldots,A_r\in {\cal V}({\rm KG}^r(\pi; \vec{s};k))$
form an edge in ${\rm KG}^r(\pi; \vec{s};k)$ where $A_1,A_2,\ldots,A_r\subseteq T$.
According to the definition of ${\rm KG}^r(\pi; \vec{s};k)$,
we have $\displaystyle\sum_{i=1}^{r}|A_i\cap P_j|\leq f_{r,{\pi}}(P_j)$
for any $m-t+1\leq j\leq m$ and also $\displaystyle\sum_{i=1}^{r}|A_i\cap L|\leq |L|$.
But
$$\begin{array}{rll}
rk=|\displaystyle\bigcup_{i=1}^{r}A_i| & =    & (\displaystyle\sum_{i=1}^{r}|A_i\cap L|)+\displaystyle\sum_{j=m-t+1}^m\displaystyle\sum_{i=1}^{r}|A_i\cap P_j|\\
             & \leq & |L|+\displaystyle\sum_{j=m-t+1}^mf_{r,{\pi}}(P_j)\\
             & =    & (kr-1)-\displaystyle\sum_{j=m-t+1}^mf_{r,{\pi}}(P_{j})+\displaystyle\displaystyle
             \sum_{j=m-t+1}^mf_{r,{\pi}}(P_j)
\end{array}$$
which is a contradiction.

Note that the size of $C=(\displaystyle\bigcup_{j=1}^{m-t}P_j)\setminus L$ is $n-M_{r,\pi}$. Set $b=\left\lceil{n-M_{r,\pi}\over r-1}\right\rceil$.
Consider a partition $(Q_1,Q_2,\ldots,Q_b)$ of $C$ such that $r-1=|Q_1|=|Q_2|=\cdots=|Q_{b-1}|\geq|Q_b|>0$.

Now, we present a proper coloring for ${\rm KG}^r(\pi; \vec{s};k)$
using $b+1$ colors.
As we mentioned, all the vertices of ${\rm KG}^r(\pi; \vec{s};k)$ that are subsets of $T$
form an independent set and therefore we can assign a color  to all of them, e.g., $b+1$.
Since every other vertex $A$ has a non-empty intersection with $C$,
we define the color of this vertex to be the minimum integer $j$ such that $A\cap Q_j\not=\varnothing$.
}\end{proof}

In the sequel, we show that $alt_2({\cal V}({\rm KG}^2(\pi; \vec{s};k)))=M_{2,{\pi}}-1.$
\begin{lem}\label{alt}
Let $k, m$ and $n$ be positive integers and $\vec{s}=(s_1,s_2,...,s_m)$ be a positive integer vector where $k\leq \displaystyle\sum_{i=1}^ms_i$.
Also, assume that $\pi=(P_1,P_2,...,P_m)$ is a partition of $[n]$, where each $P_i$ is a subset of
$|P_i|$ consecutive numbers.
If $X\in\{+1,0,-1\}^n\setminus\{(0,0,\ldots,0)\}$ and
$alt_I(X)\geq M_{2,{\pi}}$, then either $X^{+1}$ or $X^{-1}$
contains a $k$-subset $A$ of $[n]$ such that
$A$ is a vertex of ${\rm KG}^2(\pi; \vec{s};k)$.
\end{lem}
\begin{proof}{
Suppose that $X\in\{+1,0,-1\}^n\setminus\{(0,0,\ldots,0)\}$  and $alt_I(X)\geq M_{2,{\pi}}$.
Consider an alternating subsequence of non-zero signs in
$X$ of length $alt_I(X)$. Define $alt_I(X,P_j)$ to be the length of that
part of this alternating subsequence of $X$
lied in $P_j$. Also, let  $alt^+(X,P_j)$ (resp. $alt^-(X,P_j)$) be the
number of positive (resp. negative) sings of
this alternating subsequence of $X$ lied in $P_j$.
One can see that $alt_I(X,P_j)=alt^+(X,P_j)+alt^-(X,P_j)$ and $|alt^+(X,P_j)-alt^-(X,P_j)|\leq 1$.
Assume that $k^+=\displaystyle\sum_{j=1}^m \min\{s_j,\ alt^+(X,P_j)\}$
and $k^-=\displaystyle\sum_{j=1}^m \min\{s_j,\ alt^-(X,P_j)\}$.
To prove the lemma, it is enough to show that either $k^+\geq k$ or $k^-\geq k$.
Suppose therefore (reductio ad absurdum) that this is not the case.
So
$$I(X):=\displaystyle\sum_{j=1}^m \min\{s_j,\ alt^+(X,P_j)\}+
\displaystyle\sum_{j=1}^m \min\{s_j,\ alt^-(X,P_j)\}\leq 2k-2.$$
Since $alt_I(X,P_j)=alt^+(X,P_j)+alt^-(X,P_j)$ and $|alt^+(X,P_j)-alt^-(X,P_j)|\leq 1$,
one can see that
$$\min\{s_j,\ alt^+(X,P_j)\}+\min\{s_j,\ alt^-(X,P_j)\}\geq \min\{f_{2,{\pi}}(P_j),\ alt_I(X,P_j)\}.$$
Therefore,
$$\begin{array}{rll}
  2k-2 & \geq & \displaystyle\sum_{j=1}^m \min\{f_{2,{\pi}}(P_j),\ alt_I(X,P_j)\}\\
       &  =   & \displaystyle\sum_{\{j:\ alt_I(X,P_j)\geq f_{2,{\pi}}(P_j)\}}f_{2,{\pi}}(P_j)+\sum_{\{j:\ alt_I(X,P_j)< f_{2,{\pi}}(P_j)\}}alt_I(X,P_j).
\end{array}$$
This means that $\displaystyle\sum_{\{j:\ alt_I(X,P_j)\geq f_{2,{\pi}}(P_j)\}}f_{2,{\pi}}(P_j)\leq 2k-2$
and according to the definition of $M_{2,{\pi}}$ we have
$$2k-1+\displaystyle\sum_{\{j:\ alt_I(X,P_j)\geq f_{2,{\pi}}(P_j)\}}
\left(|P_j|-f_{2,{\pi}}(P_j)\right)\leq M_{2,{\pi}}$$
and therefore
$$
2k-1-M_{2,{\pi}}+\displaystyle\sum_{\{j:\ alt_I(X,P_j)\geq f_{2,{\pi}}(P_j)\}}|P_j|
\leq
\displaystyle\sum_{\{j:\ alt_I(X,P_j)\geq f_{2,{\pi}}(P_j)\}}
f_{2,{\pi}}(P_j).$$
Now, we have
$$2k-1-M_{2,{\pi}}+\displaystyle\sum_{\{j:\ alt_I(X,P_j)\geq f_{2,{\pi}}(P_j)\}}
|P_j|+\displaystyle\sum_{\{j:\ alt_I(X,P_j)< f_{2,{\pi}}(P_j)\}}alt_I(X,P_j)\leq I(X).$$
On the other hand, $I(X)\leq 2k-2$ and so
$$ 1+alt_I(X)\leq \displaystyle1+\sum_{\{j:\ alt_I(X,P_j)\geq
f_{2,{\pi}}(P_j)\}}|P_j|+\displaystyle\sum_{\{j:\ alt_I(X,P_j)< f_{2,{\pi}}(P_j)\}}alt_I(X,P_j)
\leq M_{2,{\pi}},$$
which is a contradiction.
}\end{proof}

\begin{thm}\label{main}
Let $k, n,$ and $m$ be positive integers where $k\geq 1$. Also,
assume that $\pi=(P_1,P_2,...,P_m)$ is a
partition of $[n]$ and $\vec{s}=(s_1,s_2,...,s_m)$ is a positive integer vector where $k\leq \displaystyle\sum_{i=1}^ms_i$. We have
$$\chi({\rm KG}^2(\pi;\vec{s};k))= \max\{1,n-M_{2,{\pi}}+1\}.$$
\end{thm}
\begin{proof}{
To prove this theorem, according to Lemma \ref{upper},  it is enough to show that
$\chi({\rm KG}^2(\pi;\vec{s};k))\geq n-(M_{2,{\pi}}-1)$. Define $p_i=|P_i|$, $q_0=0$, and $q_i=p_1+\cdots+p_i$.
Without loss of generality, we can suppose that for any $1\leq i\leq m$, $P_i=\{q_{i-1}+1,\ldots,q_{i}\}$.
If $M_{2,{\pi}}\geq n$, then according to Lemma \ref{upper}, the assertion holds.
Therefor, we assume that $M_{2,\pi}< n$.
Set ${\cal F}={\cal V}({\rm KG}^2(\pi;\vec{s};k))$.
In view of Lemma \ref{alt}, we have $alt_I({\cal F})\geq M_{2,{\pi}}-1$.
Consequently, by Theorem \ref{gendol},
$$\chi({\rm KG}^2(\pi;\vec{s};k))\geq n-alt_{2,I}({\cal F})\geq n-(M_{2,{\pi}}-1).$$
}\end{proof}
Note that Theorem \ref{main} provide a generalization of Lov\'asz-Kneser  Theorem \cite{MR514625}.
In fact, if we set $|P_1|=|P_2|=\cdots|P_m|=1$ and $s_1=s_2=\ldots=s_m=1$ ($\vec{s}=(1,1\ldots,1)$), then
${\rm KG}^2(\pi;\vec{s};k)={\rm KG}^2(m,k)$.

In \cite{MR1701284}, Tardif introduced the graph $K^{k,m}_t$ and called it the \emph{fractional multiple}
of the complete graph $K_t$.
This graph can be represented as follows.
The vertices of $K^{k,m}_t$ represent independent sets of size $k$ in a disjoint
union of $m$ copies of $K_t$, and
two of these are joined by an edge in $K^{k,m}_t$ if they are disjoint.
In \cite{MR2707421,MR2202051}, the chromatic number of $K^{k,m}_t$
was determined provided that $t$ is even.
It was shown that $\chi(K^{k,m}_t)=t(m-k+1)$ where $t$ is an even integer and $k\leq m$.
Although, the chromatic number of $K^{k,m}_t$ for any odd integer $t\geq 3$
was remained as an open problem. Moreover, it was conjectured in \cite{MR2707421} that $\chi(K^{k,m}_t)=t(m-k+1)$ where $t\geq 3$
is odd and $k\leq m$. In \cite{MR2202051}, it has been shown to prove $\chi(K^{k,m}_t)=t(m-k+1)$,
it suffices to show that $\chi(K^{k,m}_3)=3(m-k+1)$, since $K^{k,m}_{2t+3}$
contains a complete join of $K^{k,m}_{2t}$ and $K^{k,m}_{3}$.

Note that if we set $|P_1|=|P_2|=\cdots|P_m|=t$ and $s_1=s_2=\ldots=s_m=1$ ($\vec{s}=(1,1\ldots,1)$), then
${\rm KG}^2(\pi;\vec{s};k)=K^{k,m}_t$.
Therefore, in view of Theorem \ref{main}, we have the next corollary which
gives an affirmative answer to the aforementioned conjecture \cite{MR2707421}.
\begin{cor}
Let $t$, $k$ and $m$ be positive integers where $k\leq m$ and $t\geq 2$. Then $\chi(K^{k,m}_t)=t(m-k+1)$.
\end{cor}

Alon et al. \cite{MR857448} determined the chromatic number of Kneser hypergraphs, i.e.,
$\chi({\rm KG}^r(n,k))=\left\lceil {n-r(k-1) \over r-1}\right\rceil$.
Here, we introduce a generalization of this result.
\begin{thm}
Let $k, r, n,$ and $m$ be positive integers where $r\geq 2$ and $k\geq 1$. Also,
assume that $\pi=(P_1,P_2,...,P_m)$ is a
partition of $[n]$ and $\vec{s}=(s_1,s_2,...,s_m)$ is a positive integer vector, where for each $i\in [m]$,
$|P_i|\leq 2s_i$  and that $k\leq \displaystyle\sum_{i=1}^ms_i$. We have
$$\chi({\rm KG}^r(\pi;\vec{s};k))=\left\lceil{n-r(k-1)\over r-1}\right\rceil.$$
\end{thm}
\begin{proof}{
Set $p_i=|P_i|$, $q_0=0$, and $q_i=p_1+\cdots+p_i$.
Without loss of generality, we can suppose that for any $1\leq i\leq m$, $P_i=\{q_{i-1}+1,\ldots,q_{i}\}$.
One can check that $M_{r,\pi}=rk-1$.
Set ${\cal F}={\cal V}({\rm KG}^r(\pi;\vec{s};k))$.
In view of Theorem \ref{dolthm} and Lemma \ref{upper},
the proof is completed by showing that $alt_{r,I}({\cal F})\leq M_{r,\pi}-r+1=r(k-1)$. On the contrary,
let $X=(X^1,X^2,\ldots,X^r)\in({\cal Z}_r\cup\{0\})^n\setminus\{(0,0,\ldots,0)\}$
such that $alt_I(X)\geq r(k-1)+1$ and none of $X^i$'s contain any vertex of $KG^r(\pi;\vec{s};k)$.
Consider an alternating subsequence of non-zero terms of
$X$ of length at least $r(k-1)+1$. Define $alt_I(X,P_j)$ to be the length of that
part of this alternating subsequence
lied in $P_j$. Also, for each $\omega_i\in {\cal Z}_r$, let  $alt(X,P_j,\omega_i)$ be the
number of $\omega_i$'s of this alternating subsequence lied in $P_j$.
One can see that $alt_I(X,P_j)=\displaystyle\sum_{i=1}^{r}alt(X,P_j,\omega_i)$.
We can establish the theorem, if we prove there exists an $\omega_i\in {\cal Z}_r$ such that
$$\displaystyle\sum_{j=1}^{m}\min\left\{s_j,alt(X,P_j,\omega_i)\right\}\geq k.$$
Suppose therefore (reductio ad absurdum) that this is not the case.
Hence,
$$\displaystyle\sum_{i=1}^{r}\sum_{j=1}^{m}\min\left\{s_j,alt(X,P_j,\omega_i)\right\}\leq r(k-1).$$
Note that for any $j\in[m]$, we know $|P_j|\leq 2s_j$ and so this
implies that, for every $j\in[m]$ and $\omega_i\in {\cal Z}_r$,
we have $\min\left\{s_j,alt(X,P_j,\omega_i)\right\}=alt(X,P_j,\omega_i)$.
Consequently,
$$alt_I(X)=\displaystyle\sum_{j=1}^{m}\sum_{i=1}^{r}\min\left\{s_j,alt(X,P_j,\omega_i)\right\}\leq r(k-1),$$
which is a contradiction.
}\end{proof}
\section{Stable Kneser Graphs}
In this section, we investigate the chromatic number of $s$-stable Kneser
hypergraphs and almost $s$-stable Kneser hypergraphs.
Next proposition was proved in \cite{MR2793613} and  here we present another proof for this result.
\begin{pro}\label{upperconj}
Let $k, n, r$, and $s$ be non-negative integers where $n\geq sk$ and $s\geq r\geq 2$. We have
$$\chi({\rm KG}^r{(n, k)}_{s-stab})\leq\left\lceil{n-s(k-1)\over r-1}\right\rceil.$$
\end{pro}
\begin{proof}{
Assume $n=sq+t$ where $0< t \leq s$. Set $P_{i+1}=\{is+1,is+2,\ldots,(i+1)s\}$ for $0\leq i\leq q-1$
and $P_{q+1}=\{sq+1,sq+2,\ldots,n\}$.
Also, define $\pi=(P_1,P_2,...,P_{q+1})$ and $\vec{s}=(1,1,\ldots,1)$.
By Lemma \ref{upper}, we have
$$\chi({\rm KG}^r(\pi;\vec{s};k))\leq\left\lceil{n-M_{r,\pi}\over r-1}+1\right\rceil.$$
One can check that $M_{r,\pi}=(k-1)s+r-1$. Therefore, since
${\rm KG}^r{(n, k)}_{s-stab}$ is a subgraph of ${\rm KG}^r(\pi,\vec{s},k)$, we have
$$\chi({\rm KG}^r{(n, k)}_{s-stab})\leq\left\lceil{n-s(k-1)\over r-1}\right\rceil.$$
}\end{proof}
\begin{lem}\label{prstable}
Let $r, s$ and $p$ be positive integers where $r\geq s\geq 2$ and $p$ is a prime number.
Assume that for any $n\geq rk$, $\chi({\rm KG}^r(n,k)_{s-stab})=\left\lceil{n-r(k-1)\over r-1}\right\rceil$.
For any $n\geq prk$, we have $\chi({\rm KG}^{pr}(n,k)_{s-stab})=\left\lceil{n-pr(k-1)\over pr-1}\right\rceil$.
\end{lem}
\begin{proof}{
We endow $2^{[n]}$ with an arbitrary
total ordering $\leqslant$. Set $L=\left\lceil{n-pr(k-1)\over pr-1}\right\rceil$. In view of Proposition~\ref{upperconj}, we know that
$\chi({\rm KG}^{pr}(n,k)_{s-stab})\leq L=\left\lceil{n-pr(k-1)\over pr-1}\right\rceil$.
On the contrary, suppose that there is an integer $n\geq prk$ such that
$\chi({\rm KG}^{pr}(n,k)_{s-stab})<L$. Let $h$ be a proper $(L-1)$-coloring
of ${\rm KG}^{pr}(n,k)_{s-stab}$.
For a subset $A\subseteq [n]$ where $|A| \geq rk$, the hypergraph
${\rm KG}^r(|A|,k)_{s-stab}$ can be considered as
a subhypergraph of  ${\rm KG}^r(A,k)_{s-stab}$.
Hence, by assumption, we have
$$\chi({\rm KG}^r(A,k)_{s-stab}) \geq \chi({\rm KG}^r(|A|,k)_{s-stab})=\left\lceil{|A|-r(k-1)\over r-1}\right\rceil.$$
Consequently, if $|A|> (r-1)(L-1)+r(k-1)$, then $\chi({\rm KG}^r(A,k)_{s-stab})>L-1$.
In this case, there are $r$ pairwise disjoint vertices of
${\rm KG}^{r}(A,k)_{s-stab}$ such that $h$ assigns the same color to all of them.
Set $m=p\left((r-1)(L-1)+r(k-1)\right)+L-1$.
Now, we are going to introduce a map
$\lambda:\ (Z_p\cup\{0\})^n\setminus\{(0,0,\ldots,0)\}  \longrightarrow Z_p\times[m]$.
First, note that if $X\in (Z_p\cup\{0\})^n\setminus\{(0,0,\ldots,0)\}$
and $alt(X)>p\left((r-1)(L-1)+r(k-1)\right)$,
then there is an $1\leq i\leq p$ such that
 $|X^i|> (r-1)(L-1)+r(k-1)$.
So, $\chi({\rm KG}^r(X^i,k)_{s-stab})>L-1$; and therefore, there are $r$ pairwise disjoint
vertices $B_1,B_2,\ldots,B_r$ of ${\rm KG}^r(X^i,k)_{s-stab}$
such that $h(B_1)=\cdots=h(B_r)=c$. Set $\bar{h}(X)$ to be the greatest such a color $c$. Precisely,
$$\bar{h}(X)=\max\left\{c: \exists i, B_1,\ldots,B_r\in {X^i\choose k}_{s}, B_i\cap B_j=
\varnothing,  h(B_1)=\cdots=h(B_r)=c\right\}$$
Now, define $\lambda(X)$ as follows
\begin{itemize}
\item if $alt_I(X)\leq p\left((r-1)(L-1)+r(k-1)\right)$, set $\lambda(X)=(\omega^j,alt(X))$
such that $j$ is the least integer in $\displaystyle\bigcup_{k=1}^p X^k$.
\item if $alt_I(X)\geq p((r-1)(L-1)+r(k-1))+1$, define
$\lambda(X)=(\omega^i,p(r-1)(L-1)+pr(k-1)+\bar{h}(X))$ such that
there are $r$ pairwise disjoint vertices $B_1,B_2,\ldots,B_r$
for which $h(B_1)=h(B_2)=\cdots=h(B_r)=\bar{h}(X)$, $\displaystyle\bigcup_{k=1}^r B_k\subset X^i$
and $X^i$ is the biggest such a component of $X=(X^1,X^2,\ldots,X^p)$ respect to the ordering $\leqslant$.
\end{itemize}
One can see that $\lambda$ satisfies the conditions of $Z_p$-Tucker Lemma and therefore
we should have
$$D=p(r-1)(L-1)+pr(k-1)+(L-1)(p-1)\geq n.$$
But,
$$\begin{array}{rll}
  D &  =  & (pr-1)(L-1)+pr(k-1)\\
    &\leq & (pr-1)({n-pr(k-1)+pr-2\over pr-1}-1)+pr(k-1)\\
    &  =  & n-1,
\end{array}$$
which is a contradiction.
}\end{proof}
It was proved in \cite{MR2448565} that for $r=2^j$,
$\chi({\rm KG}^r(n,k)_{r-stab})=\left\lceil{n-r(k-1)\over r-1}\right\rceil$.
Next corollary is an immediate  consequence of this result and Lemma~\ref{prstable}.
\begin{cor}\label{alonst}
Assume that $a,k,n$ and $r$ are positive integers.
If we have $2^a|r$, then $\chi({\rm KG}^r(n,k)_{2^a-stab})=\left\lceil{n-r(k-1)\over r-1}\right\rceil$.
\end{cor}
\begin{thm}\label{2rstable}
For any positive integers $k, n$ and $r$ where $n\geq rk$, if $n\stackrel{r-1}{{\not\equiv}}k$ or $r$ is an even integer, then
$\chi({\rm KG}^r(n,k)_{2-stab})=\left\lceil{n-r(k-1)\over r-1}\right\rceil$.
\end{thm}
\begin{proof}{
In view of Corollary~\ref{alonst}, if $r$ is an even integer, then there is noting to prove.
One can see that ${\rm KG}^r(n,k)_{2-stab}$ is a subgraph of ${\rm KG}^r(n,k)$.
This implies that
$\chi({\rm KG}^r(n,k)_{2-stab})\leq \left\lceil{n-r(k-1)\over r-1}\right\rceil$.
Hence, it is sufficient to show that  if $n\stackrel{r-1}{{\not\equiv}}k$,
then $\chi({\rm KG}^r(n,k)_{2-stab})\geq \left\lceil{n-r(k-1)\over r-1}\right\rceil$.

Assume that $X\in (Z_r\cup\{0\})^n\setminus\{(0,0,\ldots,0)\}$ and $alt_I(X)\geq r(k-1)+2$.
One can see that there exists at least an $X^i$ (for some $1\leq i\leq r$)
containing some vertex of ${\rm KG}^r(n,k)_{2-stab}$.
Therefore, $alt_{r,I}({[n]\choose k}_{2})\leq  r(k-1)+1$.
By Theorem~\ref{dolthm}, we have
$$\chi({\rm KG}^r(n,k)_{2-stab})\geq \left\lceil{n-alt_{r,I}({[n]\choose k}_{2})\over r-1}\right\rceil\geq
\left\lceil{n-r(k-1)-1\over r-1}\right\rceil$$
One can check that $\left\lceil{n-r(k-1)-1\over r-1}\right\rceil=\left\lceil{n-r(k-1)\over r-1}\right\rceil$
provided that $n\stackrel{r-1}{{\not\equiv}}k$.
}\end{proof}
In view of $Z_p$-Tucker Lemma, Meunier~\cite{MR2793613} proved that,
for any positive integer $r$ and any $n\geq kp$, the chromatic number of ${\rm KG}^r(n,k)^\sim_{2-stab}$
is the same as the chromatic number of ${\rm KG}^r(n,k)$, namely that is equal to
$\left\lceil{n-r(k-1)\over r-1}\right\rceil$.
\begin{alphthm}{\rm \cite{MR2793613}}
For any $r\geq 2$, we have $\chi({\rm KG}^r(n,k)^\sim_{2-stab})=\left\lceil{n-r(k-1)\over r-1}\right\rceil$.
\end{alphthm}
\begin{proof}{
We proceed analogously to the proof of Theorem \ref{2rstable}.
Note that if  $X\in (Z_r\cup\{0\})^n\setminus\{(0,0,\ldots,0)\}$ and $alt_I(X)\geq r(k-1)+1$, then
there exists at least an $X^i$ (for some $1\leq i\leq r$) containing some
vertex of ${\rm KG}^r(n,k)^\sim_{2-stab}$.
Therefore, $alt_{r,I}({[n]\choose k}^{\sim}_{2})\leq  r(k-1)$.
By Theorem \ref{dolthm}, we have
$$\chi({\rm KG}^r(n,k)^\sim_{2-stab})\geq
\left\lceil{n-alt_{r,I}({[n]\choose k}^{\sim}_{2})\over r-1}\right\rceil\geq
\left\lceil{n-r(k-1)\over r-1}\right\rceil.$$
This completes the proof.
}\end{proof}
\section{Colorful Graphs}
We say that a graph is \emph{completely multicolored} in a coloring (\emph{colorful}) if all its vertices
receive different colors.
\begin{alphthm}{\rm \cite{MR2351519}}\label{local}
Let ${\cal F}$ be a hypergraph and $r=cd_2({\cal F})$.
Then any proper coloring of ${\rm KG}({\cal F})$ with colors $\{1,2,\ldots,k\}$
{\rm (}$k$ arbitrary{\rm )} must contain a completely multicolored complete bipartite graph
$K_{\lceil{r\over 2}\rceil,\lfloor{r\over 2}\rfloor}$ such that the $r$ different
colors occur alternating on the two sides of the bipartite graph with respect to their natural
order.
\end{alphthm}

We should mention that there are several versions of Theorem
\ref{local} in terms of some other parameters in
graphs, see  \cite{MR2087312, MR650012, MR2073516, MR2279672}.
The aforementioned theorem presents a lower bound for
{\it local chromatic number} of a graph which is the minimum number of
colors that must appear within distance $1$ of a vertex,
for more about local chromatic number, see \cite{MR837951,MR2279672}.
Next theorem provides a generalization of Theorem~\ref{local} in terms of alternation number of graphs.
\begin{thm}\label{Colorful}
Let ${\cal F}\subseteq 2^{[n]}$ be a hypergraph and $r={n-alt_2({\cal F})}$.
Then any proper coloring of ${\rm KG}({\cal F})$ with colors $\{1,2,\ldots,k\}$
{\rm (}$k$ arbitrary{\rm )} must contain a completely multicolored complete bipartite graph
$K_{\lceil{r\over 2}\rceil,\lfloor{r\over 2}\rfloor}$ such that the $r$ different
colors occur alternating on the two sides of the bipartite graph with respect to their natural
order.
\end{thm}
\begin{proof}{
Without loss of generality, we can suppose that $alt_{2}({\cal F})=alt_{2,I}({\cal F})$,
where $I$ the identity permutation on $[n]$.
First, assume that $M=alt_{2,I}({\cal F})$ is an even integer.
Consider an arbitrary total ordering $\leqslant$ on the
power set of $[n]$ that refines the partial ordering according to size.
In other words, if $|A| < |B|$, then $A\leqslant B$, and sets of the same
size can be ordered arbitrary, e.g., lexicographically.
Assume that $h$ is a proper coloring of $G={\rm KG}({\cal F})$ with $k$ colors $\{1,2,\ldots,k\}$.
Now, we construct a map
$\lambda:\{-1,0,+1\}^n\setminus \{(0,\ldots,0)\} \longrightarrow \{\pm 1, \pm 2,\ldots ,\pm m\}$ where $m=M+k$.
For $X\in\{-,0,+\}^n\setminus\{(0,0,\ldots,0)\}$, set $\lambda(X)$ as follows
\begin{itemize}
\item If $alt_I(X)\leq alt_{2,I}({\cal F})$, we define $\lambda(X)=\pm alt_I(X)$,
      where the sign is determined by the sign of the first element
      (with respect to the permutation $I$) of the longest
      alternating subsequence of $X$ (which is actually the first non-zero term of $X$).

\item If $alt_I(X)\geq alt_{2,I}({\cal F})+1$, in view of definition of $alt_I({\cal F})$,
      either $X^+$ or $X^-$ contains a member of ${\cal F}$.
      Define $c=\max\{h(F):\ F\in {\cal F}_{|X}\}$. Assume that
      $F$ is a member of ${\cal F}_{|X}$ such that $h(F)=c$. Set $\lambda(X)=\pm(h(F)+M)$, where the sign is
      positive if $F\subseteq X^+$ and negative if $F\subseteq X^-$.
\end{itemize}
It is straightforward to see that $\lambda:\{-,0,+\}^n\setminus\{(0,0,\ldots,0)\}
\longrightarrow \{\pm1,\pm2,\ldots,\pm m\}$ satisfies the conditions of Tucker-Ky Fan's Lemma.
Therefore, by Tucker-Ky Fan's Lemma, there are $n$ signed sets
$X_1\preceq X_2\preceq \cdots \preceq X_n$ such that
$\{\lambda(X_1),\ldots,\lambda(X_n)\}=\{c_1,-c_2,c_3,\ldots, (-1)^{n-1}c_n\}$
where  $1\leq c_1<c_2<\cdots<c_n\leq m$.

For any $1\leq i\leq n$, set $|X_i|=|X_i^+\cup X_i^-|$. Since $1\leq|X_1|<|X_2|<\cdots<|X_n|\leq n$, we have $|X_i|=i$.
Note that $|\lambda|$ is a monotone function; and therefore, $\lambda(X_i)=(-1)^{i-1}c_i$.
This observation concludes that $|X^+_i|=\lceil{i\over 2}\rceil$ and $|X^-_i|=\lfloor{i\over 2}\rfloor$.
In particular, $|X^+_M|={M\over 2}$ and $|X^-_M|={M\over 2}$.

Note that for $i\geq M+1$, we have $|\lambda(X_i)|=\bar{h}(X_i)+M$ and this implies that
\begin{itemize}
\item if $i$ is even, then $\bar{h}(X^-_i)=c_i$
\item if $i$ is odd, then $\bar{h}(X^+_i)=c_i$
\end{itemize}
Now, for any $i=M+2l\in\{M+1,M+2,\ldots,n\}$, there is an $F_l\in {\cal F}$
such that $F_l\subseteq X^-_i\subseteq X^-_n$ and $h(F_l)=c_{M+2l}$.
Also, for any $i=M+2l-1\in\{M+1,M+2,\ldots,n\}$, there is a $G_l\in {\cal F}$
such that $F_l\subseteq X^+_i\subseteq X^+_n$ and $h(G_l)=c_{M+2l-1}$.
Since $X^+_n\cap X^-_n=\varnothing$, the induced subgraph on vertices
$$\left\{F_1,F_2,\ldots,F_{\lfloor{n-alt_2({\cal F})\over 2}\rfloor}\right\}\bigcup\left\{G_1,G_2,\ldots,
G_{\lceil{n-alt_2({\cal F})\over 2}\rceil}\right\}$$
is a complete bipartite graph which is the desired subgraph.

Now, assume that $M$ is an odd integer. One can consider
${\cal F}$ as a subset of $2^{[n+1]}$. Note that  $alt_2({\cal F})=M+1$
is even integer and still $n-M=(n+1)-(M+1)$.
Therefore, a similar proof works when $M$ is an odd integer.
}\end{proof}
Suppose that $p$ and $q$ are positive integers where $p\geq 2q$ and $G$ is a graph.
A \emph{$(p,q)$-coloring} of $G$ is a mapping $h:{\cal V}(G)\longrightarrow \{0,1,\ldots,p-1\}$ such that
for any edge $xy\in E(G)$, we have $q\leq|h(x)-h(y)|\leq p-q$. The \emph{circular chromatic number} of $G$
is defined as follows
$$\chi_c(G)=\inf\left\{{p\over q}: G {\rm\ admits\ a}\ (p,q)-{\rm coloring}\right\}$$
It is well-known \cite{MR1815614} that $\chi(G)-1<\chi_c(G)\leq \chi(G)$.
For more on circular chromatic number, see \cite{MR1815614,MR2249284}.

The problem whether a graph has the same chromatic number and circular
chromatic number has received attention,
see~\cite{MR2601263, MR2197228, MR2279672, MR1815614,MR2249284}.
For a $t$-coloring $h$ of $G$, a cycle $C = (v_0,v_1,\ldots,v_{n-1},v_0)$
is called \emph{tight} if $h(v_{i+1})\stackrel{t}{{\equiv}}h(v_i)+1$
for $i =0,1,\ldots,n-1$, where the indices of the vertices are
modulo $n$. It is known \cite{MR1815614} that for a positive integer $t$,
$\chi_c(G)=t$ if and only if $G$ is $t$-colorable and every $t$-coloring of $G$ has a tight cycle.

Assume that ${\cal F}\subseteq 2^{[n]}$ and $\chi({\rm KG}^2({\cal F}))=n-alt_2({\cal F})$.
Previous theorem implies that, for any $\chi({\rm KG}^2({\cal F}))$-coloring
of ${\rm KG}^2({\cal F})$, there is a colorful complete bipartite graph
$K_{\lceil{\chi({\rm KG}^2({\cal F}))\over 2}\rceil,\lfloor{\chi({\rm KG}^2({\cal F}))\over 2}\rfloor}$.
This result implies that $\chi_c({\rm KG}^2({\cal F}))=\chi({\rm KG}^2({\cal F}))$
provided that $\chi({\rm KG}^2({\cal F}))$
is an even integer.

\begin{cor}\label{circular}
Assume that ${\cal F}\subseteq 2^{[n]}$ and $\chi({\rm KG}^2({\cal F}))=n-alt_2({\cal F})$. If $\chi({\rm KG}^2({\cal F}))$
is an even integer, then $\chi_c({\rm KG}^2({\cal F}))=\chi({\rm KG}^2({\cal F}))$.
\end{cor}
It has been conjectured in \cite{MR1475894} that any Kneser graph has the
same chromatic number and circular chromatic number.
This conjecture has been studied in several papers, see
\cite{MR2601263,MR2971704, MR2763055, MR1983360, MR1475894, MR2197228, MR2279672}.
Finally, Chen \cite{MR2763055} completely proved this conjecture by using
Fan's Lemma in an innovative way. Next, a shorter proof was presented in~\cite{MR2971704}.
For ${\cal F}={\cal V}({\rm KG}^2(\pi;\vec{s};k))$, in view of the proof of
Theorem \ref{main}, we have $alt_2({\cal F})=M_{2,{\pi}}-1$
and therefore,
$$\chi({\rm KG}^2(\pi;\vec{s};k))= \max\{1,n-alt_2({\cal F})\}.$$
Next corollary is a consequence of Corollary~\ref{circular}.
\begin{cor}
Let $k, n,$ and $m$ be positive integers where $k\geq 1$. Also,
assume that $\pi=(P_1,P_2,...,P_m)$ is a
partition of $[n]$ and $\vec{s}=(s_1,s_2,...,s_m)$ is a positive integer vector where $k\leq \displaystyle\sum_{i=1}^ms_i$.
If $\chi({\rm KG}^2(\pi;\vec{s};k))$ is an even integer, then
$\chi_c({\rm KG}^2(\pi;\vec{s};k))=\chi({\rm KG}^2(\pi;\vec{s};k))$.
\end{cor}
{\bf Acknowledgement:} The authors wish to express their gratitude to
Professor Carsten Thomassen for his invaluable support.
\def\cprime{$'$} \def\cprime{$'$}

\end{document}